\documentclass{classarxiv}

\usepackage{graphicx,courier,cite,amssymb,amsmath,times}
\usepackage{bm,graphicx,amssymb,amsfonts,psfrag}
\usepackage[latin1]{inputenc}
\usepackage[usenames]{color}

\usepackage[nolist]{acronym}

\def\x{{\bm x}}
\def\n{{\bm n}}

\def\I{{\bm I}}

\def\teq{\triangleq}

\newcommand{\be}{\begin{equation}}
\newcommand{\ee}{\end{equation}}
\newcommand{\ba}{\begin{array}}
\newcommand{\ea}{\end{array}}
\newcommand{\bea}{\begin{eqnarray}}
\newcommand{\eea}{\end{eqnarray}}

\def\sgn{\text{sgn}}

\newcommand{\bl}[1] {\text{\boldmath ${\lambda}$}_{#1}}

\newcommand{\vg}[1] {{\mbox{{\boldmath ${#1}$}}}}
\newcommand{\vgs}[2] {\vg{#1}_{#2}}

\newcommand{\EX}[1] {{\mathbb{E}}\left\{{#1}\right\}}

\def\1N{\vgs{1}{\!N}}

\def\nmin{n_{\text{min}}}
\def\nmax{n_{\text{max}}}
\def\Nmin{N_{\text{min}}}
\def\Nmax{N_{\text{max}}}

\def\bi{\text{\bm b}_{[i]}}

\def\teq{\triangleq}

\newcommand{\ith}[1]    {{#1}^{\underline{\text{th}}}}

\newcommand{\AZ}{\textcolor[rgb]{0,0,0}}

\def\bi{\begin{itemize}}
\def\ei{\end{itemize}}

\begin{document}

\begin{acronym}
\acro{AcR}{autocorrelation receiver} \acro{ACF}{autocorrelation
function} \acro{ADC}{analog-to-digital converter}
\acro{AWGN}{additive white Gaussian noise} \acro{BEP}{bit error
probability} \acro{BFC}{block fading channel} \acro{BPAM}{binary
pulse amplitude modulation} \acro{BPPM}{binary pulse position
modulation} \acro{BPSK}{binary phase shift keying}
\acro{BPZF}{bandpass zonal filter} \acro{CD}{cooperative
diversity} \acro{CDF}{cumulative distribution function}
\acro{CDMA}{code division multiple access}
\acro{c.d.f.}{cumulative distribution function}
\acro{ch.f.}{characteristic function} \acro{CIR}{channel impulse
response} \acro{CSCG}{circularly symmetric complex Gaussian}\acro{CSI}{channel state information}
\acro{CSIR}{channel state information at the receiver}
\acro{CSIT}{channel state information at the transmitter}
\acro{DAA}{detect and avoid} \acro{DAB}{digital audio
broadcasting} \acro{D-BLAST}{diagonally-layered space-time
architecture} \acro{E-DFE}{erasure decision feedback equalizer}
\acro{DFE}{decision feedback equalizer} \acro{DS}{direct sequence}
\acro{DS-SS}{direct-sequence spread-spectrum}
\acro{DTR}{differential transmitted-reference}
\acro{DVB-T}{digital video broadcasting\,--\,terrestrial}
\acro{ELP}{equivalent low-pass} \acro{FCC}{Federal Communications
Commission} \acro{FEC}{forward error correction} \acro{FFT}{fast
Fourier transform} \acro{FE}{fixed
energy}\acro{FH}{frequency-hopping} \acro{FH-SS}{frequency-hopping
spread-spectrum} \acro{FR}{fixed rate}\acro{GA}{Gaussian
approximation} \acro{GPS}{Global Positioning System}
\acro{HAP}{high altitude platform} \acro{i.i.d.}{independent,
identically distributed} \acro{IFFT}{inverse fast Fourier
transform} \acro{IR}{impulse radio} \acro{ISI}{intersymbol
interference} \acro{LEO}{low earth orbit}
\acro{LHS}{left-hand side}
\acro{LOS}{line-of-sight} \acro{BSC}{binary symmetric channel}
\acro{MB}{multiband} \acro{MC}{multicarrier} \acro{MF}{matched
filter} \acro{ML}{Maximum Likelihood} \acro{m.g.f.}{moment
generating function}
\acro{MGF}{moment generating function}\acro{MI}{mutual information}
\acro{MIMO}{multiple-input multiple-output}
\acro{MISO}{multiple-input single-output} \acro{MRC}{maximal ratio
combiner} \acro{MMSE}{minimum mean-square error}
\acro{MQAM}{$M$-ary quadrature amplitude modulation}
\acro{MPSK}{$M$-ary phase shift keying} \acro{MUI}{multi-user
interference} \acro{NB}{narrowband} \acro{NBI}{narrowband
interference} \acro{NLOS}{non-line-of-sight} \acro{NTIA}{National
Telecommunications and Information Administration}
\acro{OC}{optimum combining} \acro{OFDM}{orthogonal
frequency-division multiplexing} \acro{p.d.f.}{probability
distribution function} \acro{PAM}{pulse amplitude modulation}
\acro{PAR}{peak-to-average ratio} \acro{PDF}{probability density
function} \acro{PDP}{power dispersion profile}
\acro{p.m.f.}{probability mass function} \acro{PN}{pseudo-noise}
\acro{PPM}{pulse position modulation} \acro{PRake}{Partial Rake}
\acro{PSD}{power spectral density} \acro{PSK}{phase shift keying}
\acro{QAM}{quadrature amplitude modulation} \acro{QPSK}{quadrature
phase shift keying}
\acro{RHS}{right-hand side}
\acro{r.v.'s}{random variables}
\acro{r.v.}{random variable} \acro{R.V.}{random vector}
\acro{SCN}{standard condition number}
\acro{SEP}{symbol error probability} \acro{SER}{symbol error rate}
\acro{SIC}{successive interference cancellation}
\acro{SIMO}{single-input multiple-output}
\acro{SIR}{signal-to-interference ratio} \acro{SISO}{single-input
single-output} \acro{SINR}{signal-to-interference plus noise
ratio} \acro{SNR}{signal-to-noise ratio} \acro{SS}{spread
spectrum} \acro{SVD}{singular value decomposition}
\acro{TH}{time-hopping} \acro{ToA}{time-of-arrival}
\acro{TR}{transmitted-reference} \acro{BLAST}{Bell Laboratories
Layered Space-Time} \acro{V-BLAST}{Vertical Bell Laboratories
Layered Space-Time} \acro{UA}{uniform
energy allocation}\acro{UAV}{unmanned aerial vehicle} \acro{UPA}{uniform
power allocation} \acro{UWB}{ultrawide band} \acro{WLAN}{wireless
local area network} \acro{WMAN}{wireless metropolitan area
network} \acro{WPAN}{wireless personal area network}
\acro{WSN}{wireless sensor network} \acro{WSS}{wide-sense
stationary} \acro{ZF}{Zero Forcing}
\end{acronym}

\begin{acronym}
\acro{GUE}{Gaussian unitary ensemble}
\acro{GOE}{Gaussian orthogonal ensemble}
\end{acronym}

\title{On the Distribution of an Arbitrary Subset of the Eigenvalues for some Finite Dimensional Random Matrices}

\author{Marco Chiani}

\address{DEI, University of Bologna, \\
    Viale Risorgimento 2,
    40136 Bologna, ITALY
    (e-mail:  {\tt marco.chiani@unibo.it}).}

\author{Alberto Zanella}

\address{National Research Council of Italy (CNR), IEIIT, \\
    Viale Risorgimento 2,
    40136 Bologna, ITALY
    (e-mail: {\tt alberto.zanella@cnr.it}).}

\maketitle


\begin{abstract}
We present some new results on the joint distribution of an arbitrary subset of the ordered eigenvalues of complex Wishart, double Wishart, and Gaussian hermitian random matrices of finite dimensions, using a tensor pseudo-determinant operator.
Specifically, we derive compact expressions for the joint  \acl{p.d.f.} of the eigenvalues and the expectation of functions of the eigenvalues, including joint moments, for the case of both ordered and unordered eigenvalues. 
\end{abstract}

\section{Introduction}

The distribution of the eigenvalues of random matrices appears in multivariate statistics, including principal component analysis and analysis of large data sets, in physics, including  nuclear spectra, quantum
theory, atomic physics, in communication theory, especially in relation to multiple-input multiple-output systems, and in signal processing \cite{And:03,Mui:B82,Meh:04,For:B10,Handbook:11,Wint:87,Ede:88,Fos:96,Tel:99,Joh:01,ChiWinZan:J03,SmiRoySha:J03,ShiWinLeeChi:J05,CanTao:05,PenGar:09,CheMcK:12}. 
%
%
For example, the probability that the eigenvalues of a random symmetric matrix are within an interval finds application in the analysis of the stability in physics, complex networks, complex ecosystems \cite{May:72,AazEas:06,DeaMaj:08,MarMca:13,Chi:17}, for the analysis of the restricted isometry constant in compressed sensing \cite{CanTao:05,Don:06,Can:08,ElzGioChi:18}, and it is also related to the expected number of minima in random polynomials \cite{DedMal:07}.  
The distribution of the eigenvalues appears also in statistical ranking and selection theory for radar signal processing \cite{WaxKai:85,CheWicAdv:01,ChiWin:C10}, 
 in cognitive radio systems \cite{PenGarSpi:09,KoRaSeCa:11,MaMcSmNo:10,ZhMcRaWo:11,MarGioChi:J15,MarGioChi:J15a}, and for adaptive filter design \cite{Hay:08}. 

%

\AZ{Owing to the difficulties in computing the exact marginal distributions of eigenvalues, asymptotic formulas for matrices with large dimensions are often used as approximations. These approaches allow to investigate only specific subclasses of matrices. For example, the asymptotical distribution of the largest eigenvalue of Wishart matrices is known only for the uncorrelated case \cite{TraWid:09}. In the presence of correlation, the analysis is much more involved and Gaussian approximations are generally appplied \cite{DhNaSh:19}.}

For random matrices with finite dimensions (non-asymptotic analysis), the derivation of the distribution of eigenvalues is generally difficult. 
In particular, for complex matrices, which are the focus of this paper, only few results are available. Expressions for the \ac{c.d.f.} of the largest, smallest and $\ith{\ell}$ largest eigenvalue of a complex Wishart matrix have been  obtained in previous works (see for instance \cite{Kha:64,Kha:69}); however, the direct computation of the corresponding \ac{p.d.f.}'s from the \ac{c.d.f.} is not straightforward. A polynomial expression for the \ac{p.d.f.} largest eigenvalue for the uncorrelated central Wishart case was proposed in \cite{DigMalJam:J03,MaaAis:05}. The \ac{p.d.f.} of the largest eigenvalue for the case of uncorrelated noncentral Wishart was studied in \cite{AlTuLoVe:04}. Expressions for the \ac{c.d.f.} and a first order expansion for the \ac{p.d.f.} of $\lambda_{\ell}$ in the uncorrelated noncentral case were given in \cite{JiMcGaCo:J08}. The \ac{p.d.f.} of the $\ith{\ell}$ largest eigenvalue for uncorrelated central, correlated central and uncorrelated noncentral Wishart cases was also studied in \cite{KwaLeuHo:07,ChiZan:C08,ZanChiWin:J09,ZanChi:J12}.
The distribution of the largest eigenvalue and the probability that all eigenvalues are within an interval, as well as  efficient recursive methods for their numerical computation, has been  found for real and complex Wishart, multivariate Beta (also known as double Wishart or MANOVA), for the \ac{GOE} and for the \ac{GUE}  \cite{Chi:J14,Chi:J16,Chi:17}.\footnote{These matrices are also denominated, using the names of the associated weight polynomials, as Laguerre (Wishart), Jacobi (double Wishart), and Hermite (Gaussian) ensembles.} 
Expressions for the joint \ac{p.d.f.} of subsets of unordered eigenvalues of uncorrelated non central Wishart matrices were given in \cite{MarAis:J07}. 
Closed form expressions for the marginal \acp{c.d.f.} and \acp{p.d.f.} of some Hermitian random matrices, which also include Wishart matrices, were given in \cite{OrPaFo:09}. 
The \ac{MGF} of the largest eigenvalue for both uncorrelated and correlated central Wishart cases was given in \cite{McKColSmi:06}. 
Besides the finite case, approximations and asymptotics for uncorrelated Wishart and for spiked Wishart have been studied in recent literature (see e.g. \cite{Joh:01,BaiSil:06,Nad:08,JohNad:17}).  
%

The goal of the paper is to provide a unified framework for the derivation of marginal distributions, joint distribution and moments of subset of eigenvalues, for a general class of random matrices with finite size, including the \ac{GUE}, the correlated central Wishart matrices, with as a particular case the spiked Wishart, the uncorrelated noncentral Wishart matrices, and double Wishart matrices (multivariate beta). 
In particular, we generalize the results in \cite{ZanChiWin:J09,ChiZan:C08} and derive simple expressions for the joint \ac{p.d.f.} of an arbitrary subset of the eigenvalues. 
%

Indicating with $\lambda_{1} \geq \lambda_{2} \geq \ldots \geq \lambda_{M}$ the ordered nonzero eigenvalues for the mentioned random matrices,  
the contributions of the paper can be summarized as follows:
\begin{itemize}
\item We derive simple and concise expressions for the \ac{p.d.f.} of the $\ith{\ell}$ largest eigenvalue $\lambda_{\ell}$.

\item We obtain the joint distribution of $L$ arbitrary, ordered or unordered, eigenvalues. The joint distribution of two arbitrary ordered eigenvalues is a special case of this more general distribution.

\item We provide a compact expression for the expectation of statistics of the type ${\cal{S}}(\lambda_{1}, \ldots,\lambda_{M})=\prod_{i} \varphi_{i}\left(\lambda_{i}\right)$, where $\varphi_{i} \colon \mathbb{R} \to \mathbb{C}$ are arbitrary functions and $\lambda_{i}$ are the unordered eigenvalues. The joint moments of subsets of eigenvalues can be computed as a particular case.

\end{itemize}

Throughout the paper, we will use $f_X(x)$ to denote the \ac{p.d.f.} of the \ac{r.v.} X and $\EX{\cdot}$ to denote the expectation operator.
We will use bold for vectors and matrices, so that for example ${\bm x}$ denotes a vector, and ${\bm A} \in {\mathbb C}^{m \times n}$ denotes a $(m \times n)$ matrix with complex elements, $a_{i,j}$, with ${{\bm a}_j}$ denoting the $\ith{j}$ column vector of $ \bm A$. We will use $|{\bm A}|$ or $\det{\bm A}$ to denote the determinant of  ${\bm A} \in {\mathbb C}^{m \times m}$, and the superscript $(\cdot)^\dag$ for conjugation and transposition.
With ${\bm V(\bm x)}$ we indicate the Vandermonde matrix with elements $v_{i,j}=x^{i-1}_j$ and determinant $|{\bm V(\bm x)}|=\prod_{i<j}(x_{j}-x_{i})$.  
We denote by $r(x;a,b)$ the indicator function
\begin{equation*} \label{eq:rxab}
  r(x;a,b) \teq \begin{cases}
                    1 & \quad \text{if} \quad a\leq x\leq b \\
                    0 & \quad \text{elsewhere},
                \end{cases}
\end{equation*}
and with $\delta(\cdot)$ the Dirac's delta function. 

The paper is organized as follows.
The main theorems to the eigenvalue distribution of some classes of random matrices are provided in Section 2. The proof of the main result is presented in Section 3. Section 4 describes some applications of the results presented in Section 2. The results of Section 2 are also specialized in Section 5 to the case of correlated Wishart matrix. Conclusions are given in Section 6.

Throughout the paper we will generally refer to complex matrices, unless otherwise stated. 

\section{Main results}
\label{sec:mainresults}
The goal of the paper is to provide a unified framework for the derivation of marginal distributions, joint distribution of subset of eigenvalues, and moments for a general class of random matrices with arbitrary size.
To this aim,  we consider $M$ real ordered random variables ${\bm \lambda}\triangleq (\lambda_1,\lambda_2,\ldots,\lambda_{M})$ contained in the interval $(\alpha, \beta)$ with $\beta\geq \lambda_1 \geq \lambda_2 \geq\cdots \geq \lambda_{M} \geq \alpha$, whose ordered joint \ac{p.d.f.} is of the form
%
\begin{equation} \label{eq:genericpdfproddet}
f_{\bl{}} (\bm{x}) = K	\left|{\bm \Phi({\bm x})}\right| \cdot \left|{\bm \Psi}({\bm x})\right| \prod_{i=1}^{M} \xi(x_i) \,.
\end{equation}
%
In the previous equation ${\bm x}\triangleq(x_1,x_2,\ldots,x_{M})$, $K$ is a normalizing constant, $\xi(x)$ is an arbitrary function, ${\bm \Phi}\left({\bm x}\right)\in {\mathbb C}^{M \times M}$ is a matrix with elements $\phi_{i,j}=\phi_i(x_j)$, ${\bm \Psi}({\bm x}) \in {\mathbb C}^{N \times N}$ with $N \ge M$ is a matrix having elements
\begin{equation}
\label{eq:phi}
\Psi_{i,j}=\left\{\begin{array}{ll}\psi_i(x_j) &\,\,\, j=1,\ldots,M\\ 
\bar{\psi}_{i,j} &\,\,\, j=M+1,\ldots,N\end{array}\right.
\end{equation}
with $\phi_i(\cdot)$, $\psi_i(\cdot)$ arbitrary scalar functions and $\bar{\psi}_{i,j}$ arbitrary constants.
	
Expression \eqref{eq:genericpdfproddet} is of particular importance in multivariate statistical analysis as it represents the joint \ac{p.d.f.} of the eigenvalues of central Wishart or pseudo-Wishart matrices having covariance matrix with arbitrary multiplicity, noncentral Wishart with covariance matrix equal to the identity matrix,  multivariate beta (double Wishart) matrices, as well as the \ac{GUE}  \cite{ChiWinZan:J03,For:B10,Meh:04,ChiWinShi:J10,TraWid:09}. 
More precisely, some cases where the distribution of the eigenvalues is in the form \eqref{eq:genericpdfproddet} are the following. 

\def\nmin{M}
\def\Nmin{M}
\def\Nmax{N}
\def\nmax{N}
\def\x{x}
\def\X{\Theta}
\def\betareg{{\mathcal{B}}}
\def\I{{\mathcal{E}}}
\def\p{\mathsf{p}}
\def\n{\mathsf{n}}
\def\m{\mathsf{m}}

\begin{enumerate}
\item {Complex central uncorrelated Wishart matrices:} 
assume a Gaussian complex $\nmin \times n$ matrix ${\bm X}$ with \ac{i.i.d.} columns, each circularly symmetric with covariance ${\bm \Sigma= \bm I}$ (identity covariance), with $\EX{\bm X}= \bm 0$, and $n \geq \nmin$.  
%
%
 The joint \ac{p.d.f.} of the (real) ordered eigenvalues $\lambda_1 \geq \lambda_2 \ldots \geq\lambda_{\Nmin} \geq 0$ of the complex Wishart matrix ${\bm X \bm X}^H \sim {\bm {\mathcal{CW}}}_{\nmin}(n, {\bm I})$  is \cite{Ede:88,Joh:01,ChiWinZan:J03} 
\begin{equation}\label{eq:jpdfuncorrnodetcomplex}
f(x_{1}, \ldots, x_{\Nmin}) = 
K \left|{\bm V}({\bm x})\right|^2 \, \prod_{i=1}^{\Nmin}e^{-x_{i}}x_{i}^{n-\Nmin} 
\end{equation}
where $x_1 \geq x_2 \geq \cdots \geq x_{\Nmin} \geq 0$ and $K$ is a normalizing constant given by
$1/K =\prod_{i=1}^{\nmin} (n-i)! (\nmin-i)! \,.$
\item{Complex noncentral uncorrelated Wishart matrices: }
under the same hypothesis of (1), with  $\EX{\bm X}=\bm Q \ne \bm 0$, the joint \ac{p.d.f.} of the (real) ordered eigenvalues 
of the complex noncentral uncorrelated Wishart matrix ${\bm X \bm X}^H$ 
is given by \cite{JinMcKGaoCol:J08b,ZanChiWin:J09} 
\begin{equation}
\label{eq:jointPdfRicianArbitrary} f({x_1,\ldots,x_M})=K\; |{\bm W}({\bm x})|\cdot |{\bm \Upsilon }({\bm x})|\;
\prod_{i=1}^M x_i ^{n-M} e^{-x_i}
\end{equation}
where $K$ is a normalizing constant \cite{JinMcKGaoCol:J08b,ZanChiWin:J09}, the
$\ith{(i,j)}$ element of ${\bm W}({\bm x})$ is $x_j^{M-i}$, and
the $\ith{(i,j)}$ element of ${\bm \Upsilon}({\bm x})$, $\upsilon_{i,j}$,  is
\begin{equation}
\upsilon_{i,j}=\left\{
\begin{array}{ll}
\frac{{}_0 {\cal F}_1 (n-M+1,\mu_j x_i)}{(n-M)!} &\;\;\; j=1,\ldots,\nu\\
x_i^{M-j} &\;\;\; j=\nu+1,\ldots,M
\end{array}
\right.
\end{equation}
where $\nu \le M$ is the rank of ${\bm Q\; \bm Q}^H$,  $\mu_1 \ge \mu_2 \ge \cdots \ge \mu_{\nu}$ are the ordered eigenvalues of ${\bm Q\;\bm Q}^H$, and ${}_0 {\cal F}_1 (\cdot,\cdot)$ is the hypergeometric function.

\item {Hermitian Gaussian matrices (GUE):} 
the \ac{GUE} is composed of complex Hermitian random matrices with \ac{i.i.d.} ${\mathcal{CN}}(0, 1/2)$  entries on the upper-triangle, and ${\mathcal{N}}(0, 1/2)$ on the main diagonal.  
The joint distribution of the ordered eigenvalues can be written as \cite{TraWid:09}
\begin{equation} \label{eq:jpdfuncorrGUE}
f(x_1,\ldots, x_M) = K  
 \left|{\bm V}({\bm x})\right|^2 \prod_{i=1}^{M}e^{-x_{i}^2} 
\end{equation}
where $K=2^{M(M-1)/2} (\pi^{M/2} \prod_{i=1}^{M} \Gamma(i))^{-1}$ is a normalizing constant.

\item {Multivariate beta (double Wishart) matrices:} 
let ${\bm X, \bm Y} $ denote two independent complex Gaussian 
matrices, each constituted by zero mean \ac{i.i.d.} columns with common covariance. 
Multivariate analysis of variance (MANOVA) is based on the statistic of the eigenvalues of ${\bm{(A+B)}}^{-1} \bm{B}$ (beta matrix), where ${\bm A =\bm X \bm X}^H$ and ${\bm B =\bm Y \bm Y}^H$ are independent Wishart matrices.  These eigenvalues are clearly related to the eigenvalues of ${\bm{A}}^{-1} \bm{B}$ (double Wishart or multivariate beta). 
The joint distribution  of the $M$ non-null eigenvalues of a multivariate complex beta matrix in the null case can be written in the form \cite{Kha:64,Chi:17} 
\begin{equation}\label{eq:jpdfuncorrnodetcomplexMB}
f(x_{1}, \ldots, x_{M}) = K  \left|{\bm V}({\bm x})\right|^2  \, \prod_{i=1}^{M} x_{i}^{m} (1-x_{i})^n  
\end{equation}
where $m, n$ are related to the dimensions of the matrices ${\bm X, \bm Y} $, the eigenvalues are in the interval $(0,1)$ so that $1 > x_{1} \geq x_{2} \cdots \geq x_{M} > 0$, and $K$ is a normalizing constant  \cite{Kha:64,Chi:17}.
%
%
%
%

\item {Complex correlated Wishart matrices:} in Section ~\ref{sec:wishart_case} we will describe in detail the Wishart case with arbitrary correlation (including the spiked model), for which the joint distribution of the eigenvalues (see \eqref{eq:jpdfgeneralquadraticform} and \eqref{eq:jpdfcorrspiked}) has the form  \eqref{eq:genericpdfproddet}. 
\end{enumerate}

In Theorem \ref{th:det3proda} we first generalise the result \cite[Th. 2]{ChiWinZan:J03}  to cover the case of matrices having  different sizes. 

The main result of the paper is then Theorem \ref{th:marginal_joint}, which gives the marginal joint distribution of $L$ arbitrary ordered \acp{r.v.}. 

%
\begin{definition}\label{def:det3a}
For a rank $3$ tensor ${\bm A}=\left\{a_{i,j,k}\right\}_{i,j,k=1,\ldots,N}$, we define the pseudo-determinant operator ${\mathcal{T}}\left({\bm A}\right)$ as
\begin{equation}
	\label{eq:defdetdetA} {\mathcal{T}}\left({\bm A}\right) \teq
	\sum_{\bm\mu}\sgn(\bm\mu) \sum_{\bm\alpha}\sgn(\bm\alpha) \prod_{k=1}^{N}
	a_{\mu_k,\alpha_k,k}\,\,
\end{equation}
where the sums are over all possible permutations, $\bm\mu$ and $\bm\alpha$, of the integers $1,\ldots,N$.
It is worth noting that ${\mathcal{T}}\left({\bm A}\right)$ can be simplified as 
\begin{equation}
\label{eq:defdetdetAbis} {\mathcal{T}}\left({\bm A}\right) =
\sum_{\bm\mu}\sgn(\bm\mu) \det {\bm A}^{(\bm \mu)}
\end{equation}
where the $\ith{(i,j)}$ element of the matrix $ {\bm A}^{(\bm \mu)}$ is $a_{\mu_i,j,i}$. Therefore, the computational complexity of the pseudo-determinant operator is equivalent to that of $N!$ conventional determinant operators.
In particular, if the matrix $ {\bm A}^{(\bm \mu)}$ remains the same  for some permutations $\bm \mu$, the computational complexity of the operator ${\mathcal{T}}\left({\bm A}\right)$ can be strongly reduced. 
As a special case, when $a_{i,j,k}$ are independent of $k$, i.e., $a_{i,j,k}=a_{i,j,1}$, we have
\begin{eqnarray}
\label{eq:detdetAdegenappdx} {\mathcal{T}}\left({\bf A}\right) &=&
N! \det\left(\left\{a_{i,j,1}\right\}_{i,j=1\ldots,N}\right)
\end{eqnarray}
i.e., the pseudo-determinant ${\mathcal{T}}(\cdot)$ of the tensor $\left\{a_{i,j,k}\right\}_{i,j,k=1,\ldots,N}$ degenerates into $N!$ times the determinant of the matrix $\left\{a_{i,j,1}\right\}_{i,j=1\ldots,N}$.
\end{definition}

Using the above definition, we have the following theorem, which represents the generalization of \cite[Th. 2]{ChiWinZan:J03} when the integrand function is composed by the product of the determinants of two matrices having  different sizes.

\begin{theorem}\label{th:det3proda}
Given $M$ arbitrary functions $\xi_i(\cdot)$ and  two arbitrary matrices ${\bm \Phi}\left({\bm x}\right) \in {\mathbb C}^{M \times M}$, with $(i,j)$ elements $\Phi_i(x_j)$, and ${\bm \Psi}\left({\bm x}\right) \in {\mathbb C}^{N \times N}$, $N \ge  M$, with elements as in \eqref{eq:phi}, the following identity holds:
%
\begin{eqnarray}
	\label{eq:appdxprodunordtens}
	\int \ldots \int_{{\mathcal{D}}} \left| {\bm \Phi}\left({\bm x}\right)\right|
	\cdot \left| {\bm \Psi}\left({\bm x}\right)\right|
	\prod_{k=1}^{M} \, \xi_k(x_k) d{\bm x}
	={\mathcal{T}}\left({\bm C}\right)
\end{eqnarray}
where the multiple integral is over the hypercube 
$${\mathcal{D}}=\left\{a\leq x_1 \leq b, a\leq x_2 \leq b, \ldots, a\leq x_M \leq b \right\}$$  
$d{\bm x}=dx_1 \, dx_2 \cdots dx_M$  and the elements of the tensor ${\bm C}$ are
\begin{equation}
	\label{eq:ctensor}
	C_{i,j,k}=\left\{\begin{array}{ll}\displaystyle\int_{a}^b \Phi_i(x) \Psi_j(x) \xi_k(x)dx & \qquad i \le M, \,\,\,k \le M\\
	\displaystyle\int_{a}^b \Psi_j(x) \xi_k(x)dx & \qquad i>M,\,\,\,k \le M\\
	\displaystyle 0 & \qquad i < k, \,\,\,\,\,k>M \, \\
	\displaystyle\bar\Psi_{j,k} & \qquad i \ge k, \,\,\,\,\,k > M
	. \end{array}	
	\right. 
\end{equation}
%
\end{theorem}
\begin{proof}
Since the integrand function in \eqref{eq:appdxprodunordtens} does not depend on the specific values of the matrices but only on  their determinants, ${\bm \Phi}({\bm x})$ in \eqref{eq:appdxprodunordtens} can be replaced by an arbitrary  matrix, say ${\bm{\hat{ \Phi}}}\left({\bm x}\right)$, having the same determinant.  
A possible choice for the elements of ${\bm{\hat{ \Phi}}}({\bm x}) \in {\mathbb C}^{N \times N}$  is the following
\begin{equation}
\label{eq:matrixext}
\hat \Phi_{i,j}=\left\{\begin{array}{ll}
\hat \Phi_i(x_j)=\Phi_{i}(x_j) &  \qquad i \le M,\,\,\,j \le M\\
\hat \Phi_i(x_j)= 1 & \qquad i > M,\,\,\,j \le M\\
1 & \qquad i > M,\,\,\, M < j \le i\\
0 & \qquad \text{otherwise}.\\
\end{array}\right.
\end{equation}

Applying the definition \eqref{eq:matrixext}, the integral in the \ac{LHS} of \eqref{eq:appdxprodunordtens} becomes
\begin{align}
\label{eq:proof1dim1}
\int \ldots &\int_{{\mathcal{D}}} \left| {\bm{\hat{ \Phi}}}\left({\bm x}\right)\right| \cdot \left| {\bm \Psi}\left({\bm  x}\right)\right| \prod_{k=1}^M \xi_k(x_k) d{\bm x}\nonumber\\
= &\int \ldots \int_{{\mathcal{D}}} \; \left[ \sum_{\bm\sigma} \sgn\left(\bm\sigma\right) \prod_{l=1}^N \hat \Phi_{{\sigma_l},l} \right]\,\cdot \left[ \sum_{\bm \mu} \sgn\left(\bm \mu\right) \prod_{k=1}^N \Psi_{\mu_k,k} \right] \prod_{k=1}^M \xi_k(x_k) d{\bm x} \nonumber \\
= & \sum_{\bm\mu} \sgn\left(\bm\mu\right) \sum_{\bm\sigma} \sgn\left(\bm\sigma\right)\prod_{k=M+1}^N \hat \Phi_{\sigma_k,k}\,\bar{\Psi}_{\mu_k,k} \int \ldots \int_{{\mathcal{D}}} \; \prod_{k=1}^M\hat \Phi_{\sigma_k}\left(x_k\right)\Psi_{\mu_k}\left(x_k\right) \xi_k(x_k)  \; d{\bm x} \nonumber\\
= & \sum_{\bm\mu} \sgn\left(\bm\mu\right) \sum_{\bm\sigma} \sgn\left(\bm\sigma\right)\prod_{k=M+1}^N \hat \Phi_{\sigma_k,k}\,\bar{\Psi}_{\mu_k,k} \prod_{k=1}^M \int_a^b \;\hat \Phi_{\sigma_k}\left(x\right)
\Psi_{\mu_k}\left(x\right) \xi_k(x) dx \nonumber\\
= & \sum_{\bm\mu} \sgn\left(\bm\mu\right) \sum_{\bm\sigma} \sgn\left(\bm\sigma\right)\prod_{k=1}^N C_{\sigma_k,\mu_k,k} ={\cal T}\left({\bm C}\right)
\end{align}
where the elements of ${\bm C}$ are defined in \eqref{eq:ctensor}.
\end{proof}
The following Theorem gives the joint distribution of an arbitrary subset of eigenvalues. 
\begin{theorem}\label{th:marginal_joint}
The joint \ac{p.d.f.} of $L$ arbitrary ordered eigenvalues $\lambda_{i_1},\lambda_{i_2}, \ldots,\lambda_{i_L}$, with $i_1<i_2<\ldots <i_L$ with joint distribution as in \eqref{eq:genericpdfproddet} is given by
\begin{equation}\label{eq:fxi1xiLgen}
f_{\lambda_{i_1},\lambda_{i_2}, \ldots,\lambda_{i_L}}(x_{i_1},x_{i_2}, \ldots,x_{i_L})= c({i_1,i_2,\ldots,i_L}) \, K \, {\mathcal{T}}\left({\bm A}\right)
\end{equation}
where $c({i_1,i_2,\ldots,i_L})=1/\prod_{\ell=1}^{L+1}(i_\ell-i_{\ell-1}-1)!$ and the $N\times N \times N$ tensor ${\bm A}$ has elements
\begin{equation} \label{eq:aijkgeneral}
a_{i,j,k} \teq\begin{cases}
\displaystyle\int_{\alpha}^{\beta} \varsigma_i(x) \Psi_j(x)\eta_k(x) dx & k \le M\\
\displaystyle 0 & k> M,  \quad i < k \\
\displaystyle\bar{\Psi}_{j,k} & k > M, \quad i\ge k. 
\end{cases}
\end{equation}
The function $\eta_k(x)$ in the previous equation is
\begin{equation} \label{eq:etak}
\eta_k(x) \teq \begin{cases}
\delta(x-x_k)  &  \quad k\in \left\{i_1, i_2, \ldots, i_L \right\}\\
r(x; x_{i_{\varepsilon(k)}},x_{i_{\varepsilon(k)-1}})  &  \quad  \text{elsewhere}\\
\end{cases}
\end{equation}
and the segment indicator $\varepsilon(k)$ is defined as the unique integer such that $i_{\varepsilon(k)-1} \leq k < i_{\varepsilon(k)}-1$,  with the convention $i_0\teq 0 , x_0\teq \beta , i_{L+1}=M+1 , x_{L+1}=\alpha$.
\end{theorem}
\begin{proof} See Section \ref{sec:margdisteigen}. 
\end{proof}


\section{Proof of Theorem \ref{th:marginal_joint}} \label{sec:margdisteigen}
In this section we will prove Theorem \ref{th:marginal_joint}, by first deriving 
 the \ac{p.d.f.} for one eigenvalue, then for two arbitrary eigenvalues, and finally for the general case of $L$
arbitrary eigenvalues.

%
The marginal distribution of one ordered eigenvalue is obtained in the following Lemma.
\label{sec:submargdistitheigenord}
\begin{lemma}\label{th:pdfith} The \ac{p.d.f.} of the $\ith{\ell}$ ordered eigenvalue
is given by
\begin{equation}\label{eq:fxi}
f_{\lambda_\ell}(x_{\ell})=c(\ell) K \, {\mathcal{T}}\left({\bm A}\right)
\end{equation}
where
\begin{equation}
\label{eq:citheorem2}
c(\ell)\triangleq \frac{1}{(\ell-1)!(M-i)!}
\end{equation}
and the $N \times N \times N$ tensor
${\bm A=A}(x_{\ell})$ has elements
\begin{equation} \label{eq:aijkfxi}
  a_{i,j,k} \teq \begin{cases}
                   \displaystyle \int_{x_\ell}^{\beta} \varsigma_i(x) \Psi_j(x) dx &  k < \ell  \le M\\
                   \displaystyle \varsigma_i(x_\ell)  \Psi_j(x_\ell)                      &  k=\ell \le M \\
                   \displaystyle \int_{\alpha}^{x_\ell} \varsigma_i(x) \Psi_j(x) dx  &  \ell < k \le M\\
                   \displaystyle 0 & k> M,  \quad i < k \\
                   \displaystyle \bar{\Psi}_{j,k} & k > M, \quad i\ge k  
                    
                \end{cases}
\end{equation}
where
\begin{equation}
\label{eq:varsigma}\varsigma_i(x)\triangleq
\begin{cases}
\Phi_i(x) \xi(x) & i \le M\\
\xi(x) & i > M.
\end{cases}
\end{equation}
\end{lemma}
\begin{proof}
For the marginal distribution of the $\ith{\ell}$ ordered eigenvalue we
have to evaluate
\begin{equation}\label{eq:ithpdf0}
f_{\lambda_\ell}(x_\ell)=\int\cdots \int_{{\mathcal{D}(x_\ell)}}
d{\bm\lambda}^{(\ell)} 
f({\bm\lambda}^{(\ell)},x_\ell)
\end{equation}
where ${\bm\lambda}^{(\ell)}\teq (\lambda_1, \ldots, \lambda_{\ell-1},
\lambda_{\ell+1}, \ldots, \lambda_{M})$ is the vector ${\bm\lambda}$
excluding $\lambda_\ell$, and
$${\mathcal{D}(x_\ell)}=\left\{{\bm\lambda}^{(\ell)}: 
\lambda_1 \geq \cdots \lambda_{\ell-1} \geq x_{\ell} \geq \lambda_{\ell+1}
\geq \ldots \geq \lambda_{M}  \right\}.$$
The previous expression can be rewritten as
\begin{align}\label{eq:ithpdf1}
f_{\lambda_\ell}(x_\ell)&=\underbrace{\int_{x_\ell}^{\beta}d\lambda_{\ell-1}\cdots
\int_{\lambda_3}^{\beta}d\lambda_{2}
\int_{\lambda_2}^{\beta}d\lambda_{1}}_{\beta>\lambda_1\geq\lambda_2\geq\cdots\geq
x_\ell}     \nonumber \\
 & \times \underbrace{ \int_{\alpha}^{x_{\ell}}d\lambda_{\ell+1} \cdots
 \int_{\alpha}^{\lambda_{M-2}}d\lambda_{M-1}
 \int_{\alpha}^{\lambda_{M-1}}d\lambda_{M}}_{x_\ell\geq\lambda_{\ell+1}\geq\cdots\geq\lambda_{M} > \alpha } f({\bm\lambda}^{(\ell)},x_\ell).
\end{align}
Now, due to the symmetry of the function in
\eqref{eq:genericpdfproddet} we can also write
\begin{eqnarray*}\label{eq:ithpdf2}
f_{\lambda_\ell}(x_\ell)= c(\ell)\int_{x_\ell}^{\beta}\cdots
\int_{x_\ell}^{\beta}d\lambda_{1}\cdots d\lambda_{\ell-1}
 \int_{\alpha}^{x_\ell}\cdots  \int_{\alpha}^{x_\ell}d\lambda_{\ell+1}\cdots
d\lambda_{M} f({\bm\lambda}^{(\ell)},x_\ell)
\end{eqnarray*}
where $c(\ell)$ is defined in \eqref{eq:citheorem2}. To be able to use the ${\mathcal{T}}(\cdot)$ operator we must integrate $f({\bm \lambda})$ with respect to all variables (this is hypercubical integration domain). To this aim, we use the indicator function $r(x;a,b)$ defined in the introduction, 
%
that, together with the Dirac's delta function $\delta(.)$, allows
us to write
\begin{align}\label{eq:ithpdf3}  f_{\lambda_\ell}(x_\ell)&=
c(\ell) \int_{\alpha}^{\beta}\cdots
 \int_{\alpha}^{\beta}r(\lambda_1;x_\ell,\beta)\cdots r(\lambda_{\ell-1};x_\ell,\beta) \nonumber \\
 &\times \delta(\lambda_\ell-x_\ell)r(\lambda_{\ell+1};\alpha,x_\ell)\cdots
r(\lambda_{ M};\alpha,x_\ell)f({\bm \lambda})\,d{\bm \lambda}.
\end{align}
Then, by using Theorem~\ref{th:det3proda} in \eqref{eq:ithpdf3} with  $a=\alpha$, $b=\beta$, and
\begin{equation}
\xi_k(x)=\begin{cases}r(x; x_\ell,\beta)\,\xi(x) &  k < \ell\\
\delta(x-x_\ell)\,\xi(x) & k=\ell\\
r(x; \alpha, x_\ell)\,\xi(x) & \ell+1 \le k < M\\
 \end{cases}
\end{equation} we
get \eqref{eq:fxi} and \eqref{eq:aijkfxi}.
\end{proof}

The marginal joint distribution of any two ordered eigenvalues is given in the following Lemma.
\label{sec:submargdistithjtheigenord}
\begin{lemma} \label{th:jointpdfellands} The joint \ac{p.d.f.} of the $\ith{\ell}$ and $\ith{s}$ ordered
eigenvalues, $s>\ell$, is given by
\begin{equation}\label{eq:fxixj}
f_{\lambda_\ell,\lambda_s}(x_\ell,x_s)=c({\ell,s}) K {\mathcal{T}}\left({\bm A}\right)
\end{equation}
where
\begin{equation}
\label{eq:cij}
c({\ell,s})\triangleq \frac{1}{(\ell-1)!\,(s-\ell-1)!\,(M-s)!}
\end{equation}
the $N \times N \times N$ tensor ${\bm A}$ has elements
\begin{equation} \label{eq:aijk_th3}
  a_{i,j,k} \teq \begin{cases}
                    \displaystyle\int_{x_\ell}^{\beta} \varsigma_i(x) \Psi_j(x) dx  &  k < \ell  \le M\\
                    \displaystyle\varsigma_i(x_\ell)  \Psi_j(x_\ell)                       &  k=\ell \le M \\
                    \displaystyle\int_{x_s}^{x_\ell} \varsigma_i(x) \Psi_j(x) dx  &  \ell < k < s \le M\\
                    \displaystyle\varsigma_i(x_s) \Psi_j(x_s)                       &  k=s \le M \\
                    \displaystyle\int_{\alpha}^{x_s} \varsigma_i(x) \Psi_j(x) dx  &  s < k \le  M\\         
                    \displaystyle 0 & k> M,  \quad i < k \\
                    \displaystyle\bar{\Psi}_{j,k} & k > M, \quad i\ge k  
                \end{cases}
\end{equation}
and $\varsigma_i$ is defined in \eqref{eq:varsigma}.
\end{lemma}
\begin{proof} For the proof we proceed as for
Lemma~\ref{th:pdfith} and obtain
%
\begin{align}\label{eq:ithpdf_th3}
f_{\lambda_\ell,\lambda_s}(x_\ell,x_s)&= c(\ell,s) \int_{\alpha}^{\beta}\cdots
 \displaystyle \int_{\alpha}^{\beta}r(\lambda_1;x_\ell,\beta)\cdots r(\lambda_{\ell-1};x_\ell,\beta) \nonumber \\
 &\times \delta(\lambda_\ell-x_\ell)\, r(\lambda_{\ell+1}; x_s,x_\ell)\cdots r(\lambda_{s-1}; x_s,x_\ell)\, \delta(\lambda_s-x_s)\nonumber\\
 & \times r(\lambda_{s+1};\alpha,x_s)\cdots r(\lambda_{ \nmin};\alpha,x_s)f({\bm \lambda})\,d{\bm \lambda}\emph{}
\end{align}
where $c(\ell,s)$ is defined in \eqref{eq:cij}.
Using Theorem \ref{th:det3proda} 
with
\begin{equation}
\xi_k(x)=\begin{cases}r(x; x_\ell,\beta)\,\xi(x) &  k < \ell\\
\delta(x-x_\ell)\,\xi(x) & k=\ell\\
r(x; x_{s},x_{\ell}) \,\xi(x) & \ell+1 \le k < s\\
\delta(x-x_s)\,\xi(x) & k=s\\
r(x; \alpha, x_s)\,\xi(x) & s < k \le \nmin\\
 \end{cases}
 \end{equation}
 we finally obtain \eqref{eq:aijk_th3}.
\end{proof}

For the proof of the general case of Theorem \ref{th:marginal_joint}, that is, the marginal joint distribution of $L$ arbitrary ordered eigenvalues, 
we follow the same approach used for the two previous Lemmas, generalizing \eqref{eq:ithpdf_th3} to the case of $L$ variables kept fixed and integrating over the remaining $M-L$ ones. In this way we obtain \eqref{eq:fxi1xiLgen}, \eqref{eq:aijkgeneral}, and \eqref{eq:etak}
\section{Some applications of Theorems 2.1 and 2.2}
\label{sec:subexpi1genord}
\subsection{Expected value of a function of the $\ith{\ell}$ ordered eigenvalue}
\begin{theorem}
The expected value of an arbitrary function $\varphi (\cdot)$ of the $\ith{\ell}$ ordered eigenvalue is given by

\begin{equation}\label{eq:expgxiord}
\EX{\varphi (\lambda_{\ell})}=c({\ell}) K \, {\mathcal{I}}
\end{equation}
where
\begin{eqnarray} \label{eq:mathcalI}
{\mathcal{I}}
 &\teq& \int_\alpha^{\beta}\varphi (x_\ell)\, {\mathcal{T}}\left({\bm A}(x_\ell)\right)
 \,dx_\ell \nonumber\\ &=&
\sum_{\bm \mu}\sgn(\bm \mu) \sum_{\bm \alpha}\sgn(\bm \alpha)
\int_\alpha^{\beta}\varphi (x_\ell)\prod_{k=1}^{N}
 a_{\mu_k,\alpha_k,k}(x_\ell)\,dx_\ell
 \end{eqnarray}
and $a_{\mu_k,\alpha_k,k}(x_\ell)$ are defined in \eqref{eq:aijkfxi}.
\end{theorem}
\begin{proof} By direct substitution.
\end{proof}
By specializing the previous result to $\varphi (x)=x^m$ we obtain the moments of the distribution of an arbitrary ordered eigenvalue, with $\varphi (x)=r(x;0,\lambda_\ell)$ we obtain the \ac{c.d.f.}, and with $\varphi (x)=e^{\nu x}$ we get the \ac{m.g.f.}  of $\lambda_\ell$. 

Note also that in many cases the evaluation of \eqref{eq:mathcalI} does not require multidimensional numerical integration. 
For example, as shown by \eqref{eq:aijkfxi_wishart}, for Wishart matrices the functions $a_{i,j,k}(x_\ell)$ can be written in closed form.


\subsection{Probability that all eigenvalues are within the interval $[a,b]$}
\label{sec:subprobeigenab}
\begin{theorem}
The probability that all eigenvalues are within the interval
$[a,b]$ is given by

\begin{equation}\label{eq:probeigenab}
\Pr\left\{\text{All eigenvalues are} \in [a,b]\right\}=\frac{K}{M!} \,
{\cal T}\left({\bm A}\right)
\end{equation}
where the $N\times N \times N$ tensor ${\bm A}$ has elements
\begin{equation}
a_{i,j,k}=\left\{\begin{array}{ll}
\displaystyle\int_{a}^b \Phi_i(x) \Psi_j(x) \xi(x)dx & \qquad i \le M, \,\,\,k \le M\\
\displaystyle\int_{a}^b \Psi_j(x) \xi(x)dx & \qquad i>M,\,\,\,k \le M\\
0 & \qquad i < k, \,\,\,\,\,k>M \\
\bar\Psi_{j,k} & \qquad i \ge k, \,\,\,\,\,k > M
\end{array}\right.
\end{equation}
\end{theorem}
\begin{proof} For the proof we note that
\begin{eqnarray}\label{eq:probeigenab1}
\underbrace{\int\cdots \int}_{b\geq x_1 \geq \cdots \geq x_M
\geq a}
 f_{\bm \lambda}({\bm x}) d{\bm x} =
 \frac{K}{M!}\int_a^b \ldots
\int_a^b \left| {\bm \Phi}
 \left({\bm x}\right)\right| \cdot \left| {\bm \Psi}\left({\bm x}\right)\right|
\prod_{k=1}^{M} \, \xi(x_k) d{\bm x}.
\end{eqnarray}

 Then, by applying Theorem~\ref{th:det3proda} we get
\eqref{eq:probeigenab}.
\end{proof}
As a special case, if ${\bm \Psi(\bm x)} \in {\cal C}^{M \times M}$ with $\Psi_{i,j}=\Psi_i(x_j)$,\footnote{This is the case, for instance, of the joint \ac{p.d.f.} of the eigenvalues of central Wishart matrices}  the probability that all eigenvalues are within the interval
$[a,b]$ becomes
\begin{equation}\label{eq:probeigenab2}
\Pr\left\{\text{All eigenvalues are} \in [a,b]\right\}=K \,\left|{\bm B}\right|
\end{equation}
where ${\bm B} \in {\cal C}^{M \times M}$ has elements
\begin{equation} \label{eq:aijprobeigenab}
  b_{i,j} \teq \int_{a}^{b} \Phi_i(x)\Psi_j(x)\xi(x) dx \,.
\end{equation}
%
%
%
\subsection{The unordered case: marginal joint distribution of $L$ arbitrary eigenvalues.} 
\label{sec:submargdisti1i2iLgenunord}
\begin{theorem}
	The joint \ac{p.d.f.} of $L$ arbitrary unordered eigenvalues 
	$\lambda_{1},\lambda_{2}, \ldots,\lambda_{L}$ (note that due to
	symmetry we can always assume the first $L$ without loss in
	generality) is given by
	
	\begin{equation}\label{eq:fxi1xiLgenunord}
	f^{\text{(unord)}}_{\lambda_{1},\lambda_{2},
		\ldots,\lambda_{L}}(x_{1},x_{2}, \ldots,x_{L})=\frac{K}{M!}\,
	{\mathcal{T}}\left({\bm A}\right)
	\end{equation}
	where the $N\times N \times N$ tensor ${\bm A}$ has elements
	\begin{equation} \label{eq:aijkgeneralTh7}
	a_{i,j,k} =\begin{cases} 
	\displaystyle\int_{\alpha}^{\beta} \varsigma_i(x)\Psi_j(x)\,\eta_k(x)
	\,dx & k \le M\\
		0 & k > M,\quad  i < k \\
	\bar{\Psi}_{j,k} & k > M, \quad  i \ge k
	\end{cases}
	\end{equation}
	with
	\begin{equation} \label{eq:etak2}
	\eta_k(x) \teq \begin{cases}
	\delta(x-x_k)  &  \quad k \le L \\
	1  &  \quad \text{elsewhere} .
	\end{cases}
	\end{equation}
\end{theorem}
\begin{proof} For the proof we proceed similarly to the previous cases.
\end{proof}
Note that some results for the unordered case can be also found in
\cite{MaaAis:J07}.
%

\subsection{The unordered case: expected value, moments and \ac{c.d.f.} of $L$ eigenvalues}
\label{sec:subexpLgenunord}
\begin{theorem}
\label{th:expectedLunordered}
The expected value of the product of arbitrary functions $\varphi _k(\cdot)$
applied to the unordered eigenvalues is given by

\begin{equation}\label{eq:expprodetakunord}
\EX{\prod_{\ell=1}^{M}\varphi _\ell(\lambda_{\ell})}=\frac{K}{M!} \,
{\mathcal{T}}\left({\bm A}\right)
\end{equation}
where the $N\times N \times N$ tensor ${\bm A}$ has elements:
%
\begin{equation}
\label{eq:aijkgeneralTh8}
a_{i,j,k}=\left\{\begin{array}{ll} 
\displaystyle\int_{\alpha}^\beta \Phi_i(x) \Psi_j(x) \xi(x) \varphi _k(x)(x)\,dx & \qquad i \le M, \,\,\,k \le M\\
\displaystyle \int_{\alpha}^\beta \Psi_j(x) \xi(x) \varphi _k(x)\, dx & \qquad i>M,\,\,\,k \le M\\
0 & \qquad i < k, \,\,\,\,\,k>M  \\
\displaystyle\bar\Psi_{j,k} & \qquad i \ge k, \,\,\,\,\,k > M
\end{array}\right.
\end{equation}
\end{theorem}
\begin{proof} Immediate by Theorem~\ref{th:det3proda}.
\end{proof}
Special cases include the joint moments for unordered eigenvalues: 
\begin{equation}\label{eq:jointmomunord}
\EX{\lambda_{1}^{m_1}\cdots \lambda_{M}^{m_M}}
\end{equation}
obtained with $\varphi _\ell(\lambda_{\ell})=\lambda_{\ell}^{m_\ell}$ (by setting
$m_\ell=0$ for some $\ell$ we obtain the joint moments of the marginal eigenvalues). 

The joint \ac{m.g.f.} can be written as
\begin{equation}\label{eq:jointmgfunord}
M_{\bl{}}\left(\nu_1,\ldots,\nu_M \right)\teq\EX{\prod_{\ell=1}^{M} e^{\nu_\ell\lambda_{\ell}}}
\end{equation}
which can be  obtained from \eqref{eq:expprodetakunord} with $\varphi _\ell(\lambda_{\ell})=e^{\nu_\ell\lambda_{\ell}}$.

\section{Results for Complex Wishart Matrices}
\label{sec:wishart_case}
%
As previously observed, the expression for the joint \ac{p.d.f.} of the eigenvalues of complex central Wishart matrices has the same form as in \eqref{eq:genericpdfproddet}. 
To apply the results of Sections \ref{sec:mainresults} and \ref{sec:subexpi1genord} to the cases of Wishart and pseudo-Wishart matrices, the following Lemma can be used \cite{ChiWinShi:J10}.

\begin{lemma}\label{lemma: pdfquadraticforms} 

Denoting by $\bm{X}$ a complex Gaussian $(p \times n)$ random matrix with zero mean, unit variance, \ac{i.i.d.} entries and by $\bm{\Sigma}$ an $(n \times n)$ positive definite matrix, the joint \ac{p.d.f.} of the (real) nonzero ordered eigenvalues $\lambda_1 \geq \lambda_2 \geq \ldots \geq \lambda_{\nmin} \geq 0$, with $\nmin=\min(n,p)$,  of the $(p \times p)$  quadratic form $\bm{W}=\bm{X}\bm{\Sigma}\bm{X}^{\dag}$ is
\begin{equation} \label{eq:jpdfgeneralquadraticform}
f_{\bl{}} (x_{1}, \ldots, x_{\nmin})
    = K
    \left|{\bm V(\bm x)}\right| \cdot
        \left|{\bm G}({\bm x,\bm \phi})\right| \prod_{i=1}^{\nmin} \xi(x_i)
\end{equation}
where $\xi(x)=x^{p-\nmin}$, ${\bm V(\bm x)}$ is the ($\nmin \times \nmin$) Vandermonde matrix with elements $v_{i,j}=x^{i-1}_j$. 
The constant $K$ is given by 
\begin{equation}\label{eq:Kquadformgen}
K= \frac{(-1)^{p(n-\nmin)}}{\Gamma_{(\nmin)}(p)} \,
\frac{\prod_{i=1}^{L}\phi_{(i)}^{m_i p}}{ \prod_{i=1}^{L}
\Gamma_{(m_i)}(m_i)
\prod_{i<j}\left(\phi_{(i)}-\phi_{(j)}\right)^{m_i m_j}}
\end{equation}
where $\Gamma_{(m)}(n)\triangleq \prod_{i=1}^{m}(n-i)!$ and  $\phi_{(1)} > \phi_{(2)} \ldots > \phi_{(L)}$ are the $L$ distinct eigenvalues of ${\bm \Sigma}^{-1}$, with associated multiplicities $m_1, \ldots , m_L$ such that $\sum_{i=1}^{L} m_i=n$. 

The ($n\times n$) matrix ${\bm G}({\bm x,\bm\phi})$ has elements
\begin{equation} \label{eq:gtilde}
{g}_{i,j}=\left\{
  \begin{array}{cll}
    g_{i}(x_j)&=\left(-x_j \right)^{d(i)} e^{-\phi_{(e(i))} x_j} & \,\,\,j=1,\ldots, \nmin \\
    \bar g_{i,j}&=\left[n-j\right]_{d(i)} \phi_{(e(i))}^{n-j-d(i)}&\,\,\, j=\nmin+1,\ldots,n \\
  \end{array}
  \right.
\end{equation}
where $[a]_n\teq a(a-1)\cdots (a-n+1)$,  $e(i)$ denotes the unique
integer such that
$$m_1+\ldots+m_{e(i)-1}< i \leq m_1 +\ldots+m_{e(i)} $$
and 
$$d(i)=\sum_{k=1}^{e(i)}m_k -i .$$
\end{lemma}
%

It can be checked that the uncorrelated case \eqref{eq:jpdfuncorrnodetcomplex} is the special case of \eqref{eq:jpdfgeneralquadraticform} for ${\bm \Sigma}={\bm I}$. 

Another interesting special case is when ${\bm \Sigma}$ is spiked, i.e.,  with $\sigma_1>\sigma_2=\sigma_3=\sigma_4=\cdots=\sigma_n$. For this spiked model correlation we have the following result.
\begin{lemma}\label{lemmaCWspiked}
Let ${\bm W} \sim {\mathcal{CW}}_M(n, \bm{\Sigma})$ be a complex Wishart matrix, $n\geq M$. 
Denote $\sigma_1 > \sigma_2 = \ldots = \sigma_{M} > 0$ the ordered eigenvalues of $\bm \Sigma$ (spiked covariance matrix). 
Then, the joint \ac{p.d.f.} of the ordered eigenvalues of  ${\bm W}$ is  
\begin{align}\label{eq:jpdfcorrspiked}
f_{\bl{}}(x_{1}, \ldots, x_{\nmin})
    &= K 
    \left|\bm{E}\left({\bm x},\bm{\sigma}\right)\right|
        \cdot \prod_{i<j}^{\nmin} (x_i-x_j)
        \cdot \prod_{i=1}^{\nmin} x^{n-\nmin}_i 
\end{align}
where $\bm{E}\left({\bm x},\bm{\sigma}\right)$ has elements 
\begin{align*}
\label{eq:eijcomplexcorr}
e_{i,j}=\left\{
  \begin{array}{ll}
    \displaystyle e^{-x_i/\sigma_1}  & j=1\\
   \displaystyle x_i^{\nmin-j} e^{-x_i/\sigma_2}  & j=2, \ldots,\nmin
  \end{array}
\right.
\end{align*}
and
\begin{equation*}
\label{eq:ksigmaspiked}
\frac{1}{K} =  \sigma_1^{n-\nmin+1} \sigma_2^{(n-1)(\nmin-1)} (\sigma_1-\sigma_2)^{\nmin-1} \, \prod_{i=1}^\nmin (n-i)! \prod_{\ell=2}^{\nmin-2} \ell ! 
\end{equation*}
\end{lemma}
\begin{proof}
This is a particular case of Lemma~\ref{lemma: pdfquadraticforms}.
\end{proof}

Using the \ac{p.d.f.} expression in lemma \ref{lemma: pdfquadraticforms}, The results of Sections \ref{sec:mainresults} and \ref{sec:subexpi1genord} can be easily specialized to the Wishart (or pseudo-Wishart) case. In particular, $M$, $N$,
${\bm \Phi({\bm x})}$, and ${\bm \Psi(\bm x)}$ in Theorems \ref{th:pdfith}-\ref{th:expectedLunordered}  must be replaced by $\nmin$, $p$, ${\bm V(\bm x)}$ and ${\bm G(x,\phi)}$ in Lemma \ref{lemma: pdfquadraticforms}.

\noindent For example, the elements of tensors in \eqref{eq:aijkfxi} of Lemma \ref{th:pdfith} and \eqref{eq:aijk_th3} of Lemma \ref{th:jointpdfellands}  can be written for the Wishart matrices as follows. 

\begin{itemize}
\item Tensor elements for the distribution of one eigenvalue for Wishart matrices, \eqref{eq:aijkfxi} of Lemma \ref{th:pdfith}:
  \begin{equation} \label{eq:aijkfxi_wishart}
  \begin{array}{c}
\quad  a_{i,j,k} \teq \begin{cases}
         (-1)^{d(j)} \phi_{(e(j))}^{-\omega_{i,j}}\Gamma\left(\omega_{i,j},x_{\ell} \phi(e(j))\right)  &  k < \ell  \le \nmin\\[\medskipamount]
                    x_{\ell}^{p-\nmin+\zeta_{i}}  \left(-x_{\ell} \right)^{d(j)} e^{-\phi_{(e(j))}}                      &  k=\ell \le \nmin \\[\medskipamount]
         (-1)^{d(j)} \phi_{(e(j))}^{-\omega_{i,j}}\gamma\left(\omega_{i,j},x_{\ell} \phi(e(j))\right)  &  \ell < k \le \nmin\\[\medskipamount] 
         0 & k> \nmin,  \quad i < k \\[\medskipamount]
         \left[n-j\right]_{d(i)} \phi_{(e(i))}^{n-j-d(i)} & k > \nmin, \quad i\ge k  \\[\medskipamount]     
                \end{cases}
                \end{array}
   \end{equation}
   
%
\item Tensor elements for the distribution of two eigenvalues for Wishart matrices, \eqref{eq:aijk_th3} of Lemma \ref{th:jointpdfellands}:

\begin{equation} \label{eq:aijk_th3_wishart}
\begin{array}{c}
\quad  a_{i,j,k} \teq \begin{cases}
                     (-1)^{d(j)} \phi_{(e(j))}^{-\omega_{i,j}}\Gamma\left(\omega_{i,j},x_{\ell} \phi(e(j))\right)  &  k < \ell  \le \nmin\\[\medskipamount]
                   x_{\ell}^{p-\nmin+\zeta_{i}}  \left(-x_{\ell} \right)^{d(j)} e^{-\phi_{(e(j))}}                      &  k=\ell \le \nmin \\[\medskipamount]
                   (-1)^{d(j)} \phi_{(e(j))}^{-\omega_{i,j}} \Gamma\left(\omega_{i,j},x_{s} \phi(e(j)),x_{\ell} \phi(e(j))\right)  &  \ell < k < s \le \nmin\\[\medskipamount]
                    x_{s}^{p-\nmin+\zeta_{i}}  \left(-x_{s} \right)^{d(j)} e^{-\phi_{(e(j))}}                     &  k=s \le \nmin \\[\medskipamount]
                    (-1)^{d(j)} \phi_{(e(j))}^{-\omega_{i,j}}\gamma\left(\omega_{i,j},x_{s} \phi(e(j))\right)   &  s < k \le  \nmin\\
                    0 & k> \nmin,  \quad i < k \\
                    \left[n-j\right]_{d(i)} \phi_{(e(i))}^{n-j-d(i)} & k > \nmin, \quad i\ge k                  
                \end{cases}\\
                \end{array}
\end{equation}
where \begin{equation}
\label{eq:zeta}\zeta_i\triangleq
\begin{cases}
i-1 & i \le \nmin\\
0 & i > \nmin
\end{cases}
\end{equation}
$\omega_{i,j}\triangleq p-\nmin+\zeta_i+d(j)+1$, the upper and lower incomplete Gamma functions are indicated as $\Gamma(\cdot,\cdot)$ and $\gamma(\cdot,\cdot)$, respectively, 
and 
$\Gamma(n,x_1,x_2)\triangleq \Gamma(n, x_1)-\Gamma(n, x_2)$ \cite{AbrSte:B70}. 

\item Tensor elements for the distribution of one eigenvalue for Wishart matrices with spiked covariance matrix, \eqref{eq:aijkfxi} of Lemma \ref{th:pdfith}:
\begin{equation} \label{eq:aijkfxi_spiked_quat}
\begin{array}{c}
\quad  a_{i,j,k} \teq (-1)^{i-1}\begin{cases}
 \sigma_1^{\theta_i} \Gamma(\theta_i,x_\ell/\sigma_1) & k < \ell\quad j=1\\
\sigma_2^{\varrho_{i,j}}\Gamma(\varrho_{i,j},x_\ell/\sigma_2) & k < \ell\quad j>1\\
 x_\ell^{\theta_i-1}  e^{x_\ell/\sigma_1} & k=\ell\quad j=1\\
 x_\ell^{\varrho_{i,j}-1}  e^{-x_\ell/\sigma_2} & k=\ell\quad j>1\\
\sigma_1^{\theta_i} \gamma(\theta_i,x_\ell/\sigma_1) & k > \ell\quad j=1\\
 \sigma_2^{\varrho_{i,j}}\gamma(\varrho_{i,j},x_\ell/\sigma_2)  & k > \ell\quad j>1\\
\end{cases}\\
\end{array}
\end{equation}
where  $\theta_i \triangleq n-\nmin+i$ and $\varrho_{i,j} \triangleq n+i-j$.
\item Tensor elements for the distribution of two eigenvalues for Wishart matrices with spiked covariance matrix, \eqref{eq:aijk_th3} of Lemma \ref{th:jointpdfellands}:

\begin{equation} \label{eq:aijk_spiked2_esa}
a_{i,j,k} \teq (-1)^{i-1} \begin{cases}
\displaystyle  \sigma_1^{\theta_i}\Gamma(\theta_i,x_\ell/\sigma_1) &  k < \ell  \le M \quad j=1\\
\displaystyle \sigma_2^{\varrho_{i,j}}\Gamma(\varrho_{i,j},x_\ell/\sigma_2) &  k < \ell  \le M \quad j>1\\
\displaystyle   x_\ell^{\theta_i-1}  e^{-x_\ell/\sigma_1}                         &  k=\ell \le M \quad j=1 \\
\displaystyle   x_\ell^{\varrho_{i,j}-1}   e^{-x_\ell/\sigma_2}                       &  k=\ell \le M \quad j>1 \\
\displaystyle  \sigma_1^{\theta_i}\Gamma(\theta_i,x_s/\sigma_1,x_\ell/\sigma_1) &  \ell < k < s \le M \quad j=1 \\
\displaystyle \sigma_2^{\varrho_{i,j}}\Gamma(\varrho_{i,j},x_s/\sigma_2,x_\ell/\sigma_2) &  \ell < k < s \le M \quad j>1 \\
\displaystyle   x_s^{\theta_i-1}  e^{-x_s/\sigma_1}                       &  k=s \le M \quad j=1 \\
\displaystyle   x_s^{\varrho_{i,j}-1}  e^{-x_s/\sigma_2}                       &  k=s \le M \quad j>1 \\
\displaystyle  \sigma_1^{\theta_i}\gamma(\theta_i,x_s/\sigma_1) &  s < k \le  M \quad j=1 \\
\displaystyle  \sigma_2^{\varrho_{i,j}}\gamma(\varrho_{i,j},x_s/\sigma_2) &  s < k \le  M \quad j>1. \\
\end{cases}
\end{equation}

\end{itemize}

\subsection{Numerical examples}
We present some numerical examples related to the \ac{p.d.f.} of the $\ith{i}$ largest eigenvalue and the joint \ac{p.d.f.} of the $\ith{i}$ and $\ith{j}$ ordered eigenvalues of a central Wishart matrix. 
For the sake of conciseness we only show results for the uncorrelated and for the spiked correlated Wishart cases. 

Fig.~\ref{fig:pdflambdai} shows the \acp{p.d.f.} of the various ordered eigenvalues $\lambda_1,\ldots,\lambda_4$, of the uncorrelated Wishart matrix with $M=4$ and $n=5$. The curves have been  obtained from Lemma~\ref{th:pdfith} and \eqref{eq:aijkfxi}, where, starting from \eqref{eq:jpdfuncorrnodetcomplex}  we get for the uncorrelated Wishart $\Psi_i(x)=\Phi_i(x)=x^{i-1}$ and $\xi(x)=x^{n-M} e^{-x}$. Note that the integrals in  \eqref{eq:aijkfxi} are in this case expressible in terms of gamma functions. 

To show the effect of correlation, in Fig.~\ref{fig:pdflambdaispiked1} the \acp{p.d.f.} of the various ordered eigenvalues $\lambda_1,\ldots,\lambda_4$ of the spiked correlated Wishart matrix with $M=4$ and $n=5$ are reported. The correlation matrix here has eigenvalues $\sigma_{1}=10, \sigma_{2}=\sigma_{3}=\sigma_{4}=1$. One of the effects of a spiked correlation is to increase the expected value and variance of the largest eigenvalue, as can be seen from Fig.~\ref{fig:pdflambdaispiked1}.  

In Fig.~\ref{fig:pdflambdainmin6nmax10} and Fig.~\ref{fig:pdflambdaispikednmin6nmax10} we report the eigenvalues distribution for $M=6$ and $n=10$, with and without correlation. 
To allow a comparison with the uncorrelated case the same scale is kept. For this reason, only a part of the left tail of the distribution of the largest eigenvalue is visible.

\AZ{One of the possible applications of the expression for the marginal distribution of single eigenvalues is in the field of the performance analysis of communications systems characterized by the presence of \ac{MIMO} systems, characterized by the presence of multiple antennas at both transmit and receive side. More specifically, in the \ac{MIMO} scheme denoted as \ac{SVD} \ac{MIMO}, the symbol error probability associated at each eigen-channel depends on the value of the correspondent eigenvalue of the \ac{MIMO} channel matrix, which is typically modeled as a Wishart \cite{Tel:99}. 
}

In Fig.~\ref{fig:pdflambdainmin6GUE} we report the eigenvalues distribution for $M=6$ for the Gaussian unitary ensemble. The curves have been  obtained from Lemma~\ref{th:pdfith} and \eqref{eq:aijkfxi}, where, starting from \eqref{eq:jpdfuncorrGUE}  we get for the GUE $\Psi_i(x)=\Phi_i(x)=x^{i-1}$ and $\xi(x)=e^{-x^2}$. Note that even in this case the integrals in  \eqref{eq:aijkfxi} are expressible in terms of gamma functions.  
%

With reference to the joint distribution of two eigenvalues, we report in Figs.~\ref{fig:pdflambda12}-\ref{fig:pdflambda34} the joint \ac{p.d.f.} for all possible couples of eigenvalues in the case of uncorrelated Wishart matrix with $M=4$ and $n=5$. The surfaces have been obtained from Lemma~\ref{th:jointpdfellands} and \eqref{eq:aijk_th3} with $\Psi_i(x)=\Phi_i(x)=x^{i-1}$ and $\xi(x)=x^{n-M} e^{-x}$. Even in this case the integrals in  \eqref{eq:aijk_th3} can be expressed in terms of gamma functions.

\AZ{Note that, as shown in \eqref{eq:defdetdetAbis}, the computation of the operator  ${\mathcal{T}}\left({\bm A}\right)$ requires the evaluation of $N!$ determinants of $(N \times N)$ matrices; this number can make the operation impractical for large values of $N$.  Furtunately, when dealing with the evaluation of the distribution of subsets of the eigenvalues,  the elements of the rank 3 tensor present some regularity patterns that can be exploited to simplify the evaluation of  ${\cal T}({\bf A})$. In particular, the matrix to be evaluated does not change for some kinds of permutations; therefore, we can group the N! permutations so that each group contains permutations that provide the same determinant. This latter consideration leads to a  significant reduction of the number of determinants to be computed. 	More specifically, for the derivation of the p.d.f. of the $\ith{\ell}$ ordered eigenvalue, the number of determinants reduces from $N!$ to 
\begin{equation}
\frac{N!}{(\ell-1)! (M-\ell)!}.
\end{equation}
For the case of the derivation of the joint p.d.f. of the $\ith{\ell}$ and $\ith{s}$ (with $s > \ell$) ordered eigenvalues, the number of permutations reduces to 
\begin{equation}
\frac{N!}{(\ell-1)! (s-\ell-1)! (M-s)!}.
\end{equation}
The procedure can be easily generalized to the case of joint p.d.f. of $L$ ordered eigenvalues; in this case the number of the determinants reduces to
\begin{equation}
	\frac{N!}{(M-i_L)!\,\, \prod_{m=1}^{L} (i_{m}-i_{m-1}-1)!}
\end{equation}
where $i_0=0$. It is worth noting that these results  hold for all the matrices whose joint p.d.f. of the eigenvalues takes the form in (2.1); therefore,  this approach can be applied to a very general class of matrices, like, for instance, Wishart, Hermitian Gaussian, Multivariate beta.
}


\section{Conclusions}
%
In this paper, we focused on the random matrices whose joint distribution of the eigenvalues can be written in the form \eqref{eq:genericpdfproddet} and proposed a unified framework for the derivation of the marginal distribution of the eigenvalues, some related moments, and the joint distribution of an arbitrary subset of ordered eigenvalues.  The results can be applied to the case of both central (uncorrelated or correlated, including the spiked model) and noncentral uncorrelated Wishart matrices, double Wishart (beta) matrices, as well as to \ac{GUE}, and can be used to address many aspects of interest for wireless communications, radar signal processing, and physics. 

\section*{Acknowledgments}
This work was supported in part by the Italian Ministry for Education, University and Research (MIUR) under the program Dipartimenti di Eccellenza (2018-2022).


\bibliographystyle{IEEEtran}
\bibliography{IEEEabrv,BibEIGEN,BibMIMO,MIMO,MyBooks,BiblioMCCV,RandomMatrix}

\newpage


\begin{figure}[h]
\centerline{\includegraphics[width=0.8\columnwidth,draft=false]{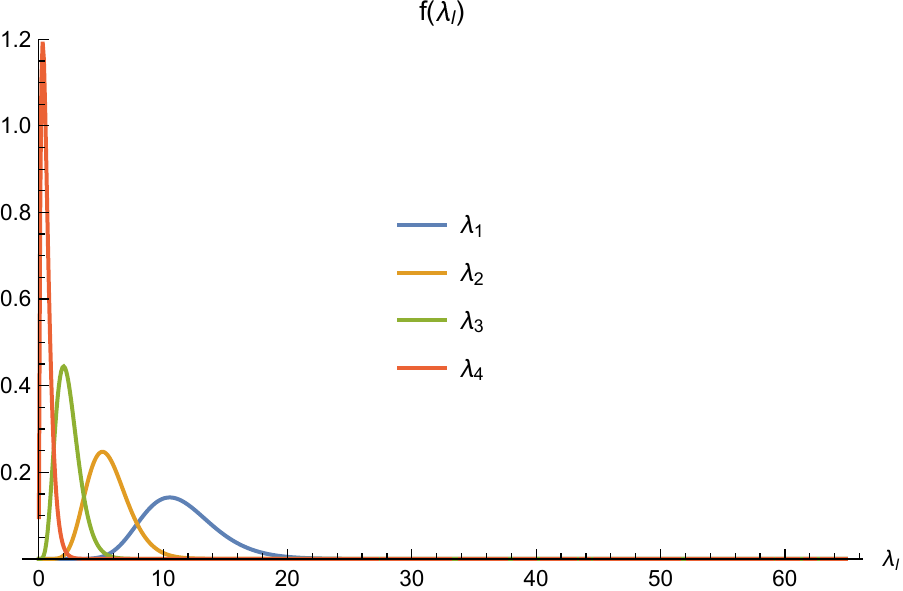}}
\caption{Marginal \acl{p.d.f.} of each ordered eigenvalue for the uncorrelated central Wishart matrix with $M=4$ and $n=5$. The correlation matrix here has eigenvalues $\sigma_{1}= \sigma_{2}=\sigma_{3}=\sigma_{4}=1$.}
\label{fig:pdflambdai}
\end{figure}

\begin{figure}[h]
\centerline{\includegraphics[width=0.8\columnwidth,draft=false]{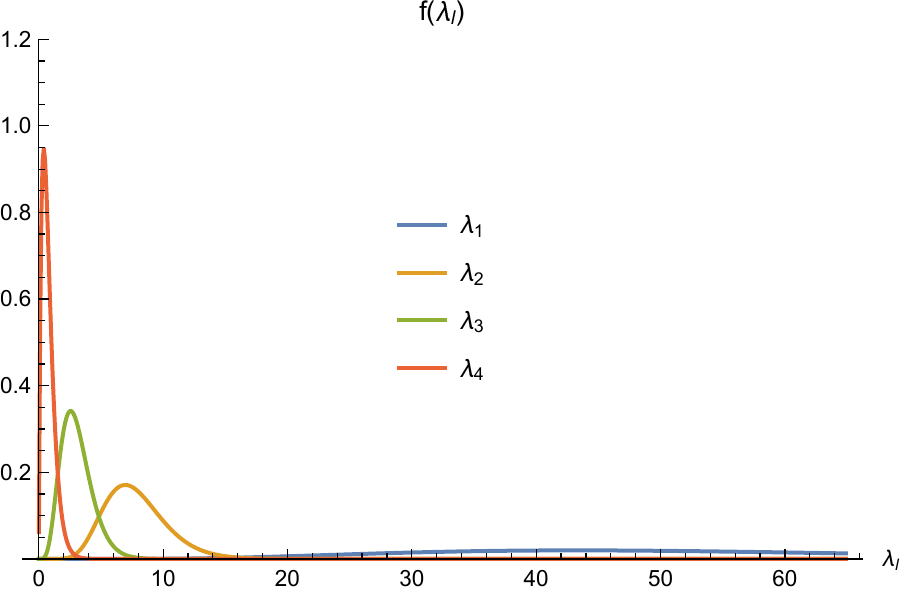}}
\caption{Marginal \acl{p.d.f.} of each ordered eigenvalue for the spiked correlated central Wishart matrix with $M=4$ and $n=5$. The correlation matrix here has eigenvalues $\sigma_{1}=10, \sigma_{2}=\sigma_{3}=\sigma_{4}=1$. }
\label{fig:pdflambdaispiked1}
\end{figure}

\begin{figure}[h]
\centerline{\includegraphics[width=0.8\columnwidth,draft=false]{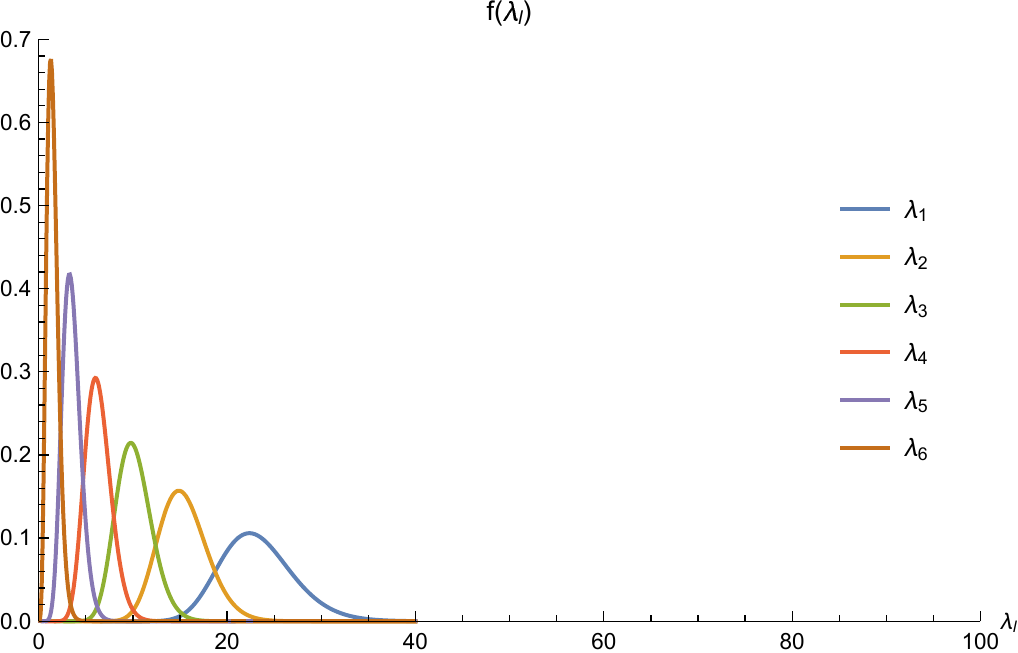}}
\caption{Marginal \acl{p.d.f.} of each ordered eigenvalue for the uncorrelated central Wishart matrix with $M=6$ and $n=10$. The correlation matrix here has eigenvalues $\sigma_{i}=1, \quad i=1, \ldots, 6$.}
\label{fig:pdflambdainmin6nmax10}
\end{figure}

\begin{figure}[h]
\centerline{\includegraphics[width=0.8\columnwidth,draft=false]{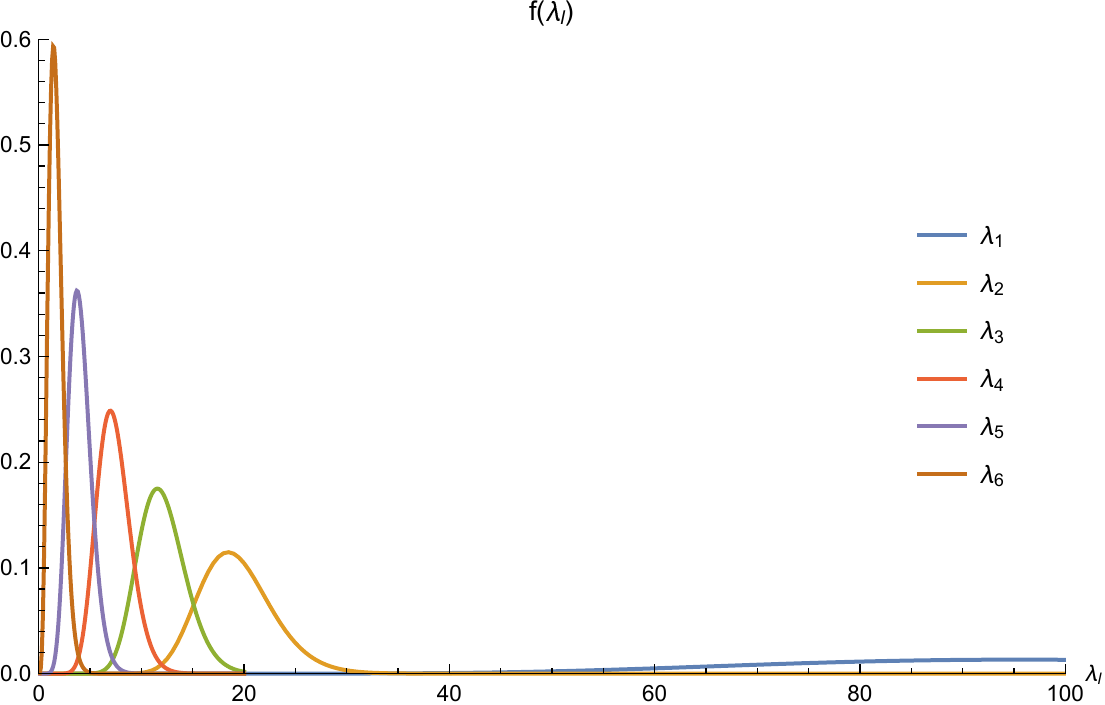}}
\caption{Marginal \acl{p.d.f.} of each ordered eigenvalue for the spiked correlated central Wishart matrix with $M=6$ and $n=10$. For $\lambda_{1}$ just the left tail is visible on this scale. The correlation matrix here has eigenvalues $\sigma_{1}=10$, $\sigma_{i}=1, \quad i=2, \ldots, 6$. }
\label{fig:pdflambdaispikednmin6nmax10}
\end{figure}

\begin{figure}[h]
\centerline{\includegraphics[width=0.8\columnwidth,draft=false]{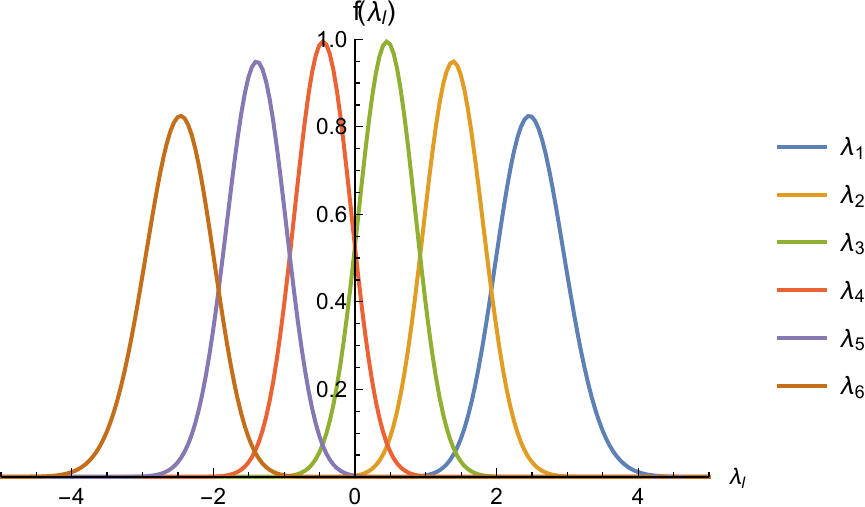}}
\caption{Marginal \acl{p.d.f.} of each ordered eigenvalue for the GUE with $M=6$. }
\label{fig:pdflambdainmin6GUE}
\end{figure}

\begin{figure}[h!]
\psfrag{yy}{}
\psfrag{l1}{$\lambda_1$}
\psfrag{l2}{$\lambda_2$}
\psfrag{l3}{$\lambda_3$}
\psfrag{l4}{$\lambda_4$}
\centerline{\includegraphics[width=7.5cm,draft=false]
    {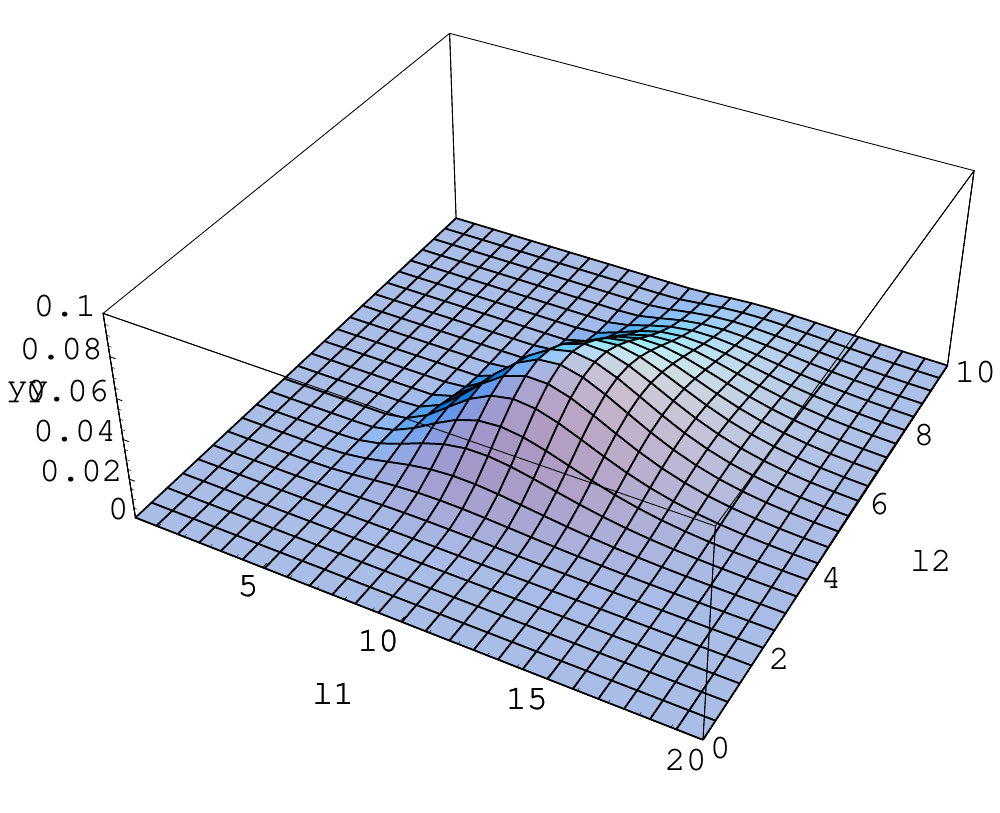}}
\caption{Joint probability density function of $\lambda_1$ and $\lambda_2$ of the uncorrelated central Wishart matrix with $M=4$ and $n=5$.} 
\label{fig:pdflambda12}
\end{figure}

\begin{figure}[h!]
\psfrag{yy}{}
\psfrag{l1}{$\lambda_1$}
\psfrag{l2}{$\lambda_2$}
\psfrag{l3}{$\lambda_3$}
\psfrag{l4}{$\lambda_4$}
\centerline{\includegraphics[width=7.5cm,draft=false]
    {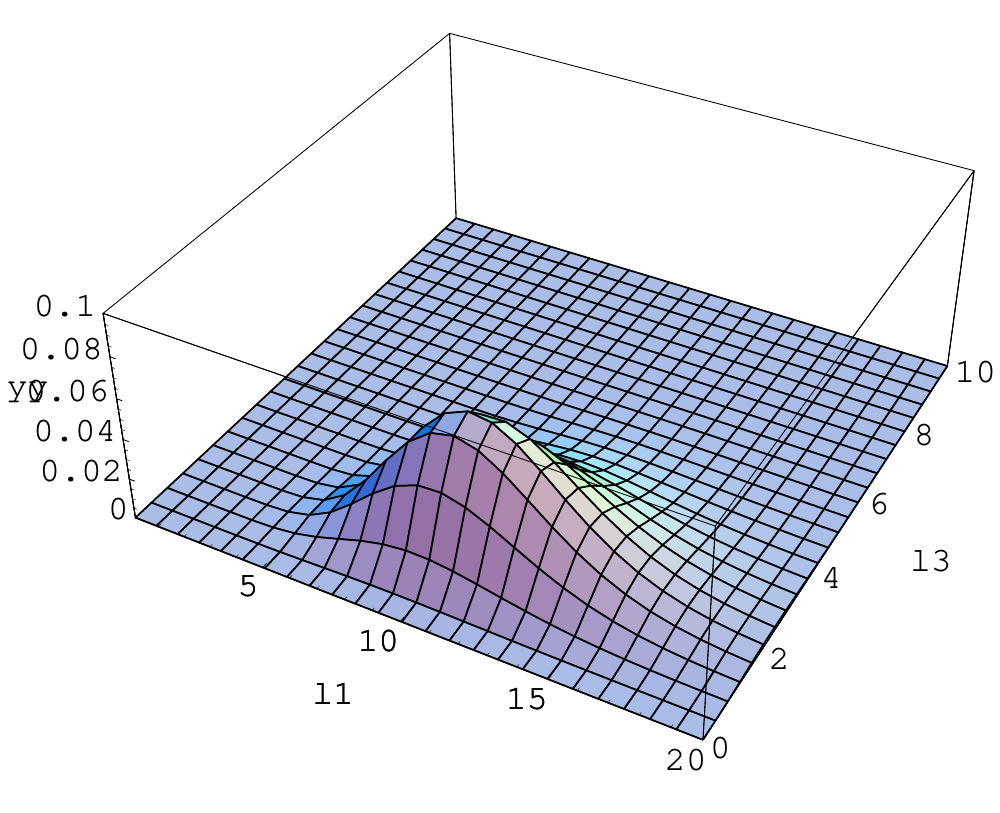}}
\caption{Joint probability density function of $\lambda_1$ and $\lambda_3$ of the uncorrelated central Wishart matrix with $M=4$ and $n=5$.} 
\label{fig:pdflambda13}
\end{figure}

\begin{figure}[h!]
\psfrag{yy}{}
\psfrag{l1}{$\lambda_1$}
\psfrag{l2}{$\lambda_2$}
\psfrag{l3}{$\lambda_3$}
\psfrag{l4}{$\lambda_4$}
\centerline{\includegraphics[width=7.5cm,draft=false]
    {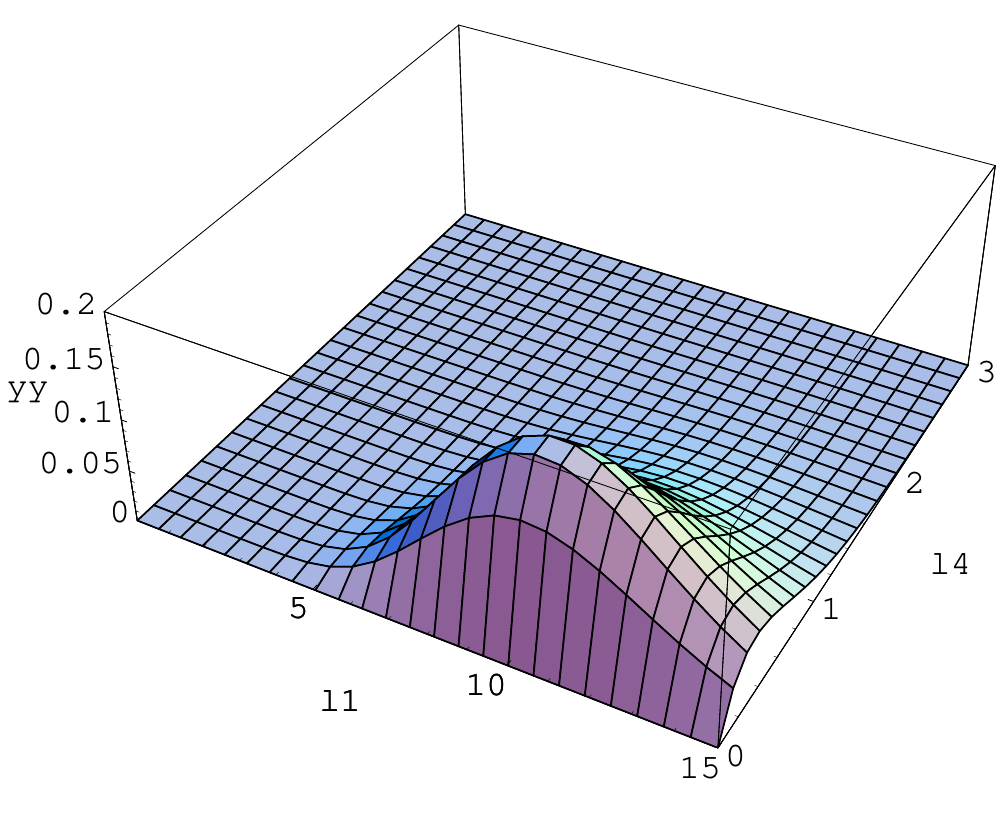}}
\caption{Joint probability density function of $\lambda_1$ and $\lambda_4$ of the uncorrelated central Wishart matrix with $M=4$ and $n=5$.} 
\label{fig:pdflambda14}
\end{figure}

\begin{figure}[h!]
\psfrag{yy}{}
\psfrag{l1}{$\lambda_1$}
\psfrag{l2}{$\lambda_2$}
\psfrag{l3}{$\lambda_3$}
\psfrag{l4}{$\lambda_4$}
\centerline{\includegraphics[width=7.5cm,draft=false]
    {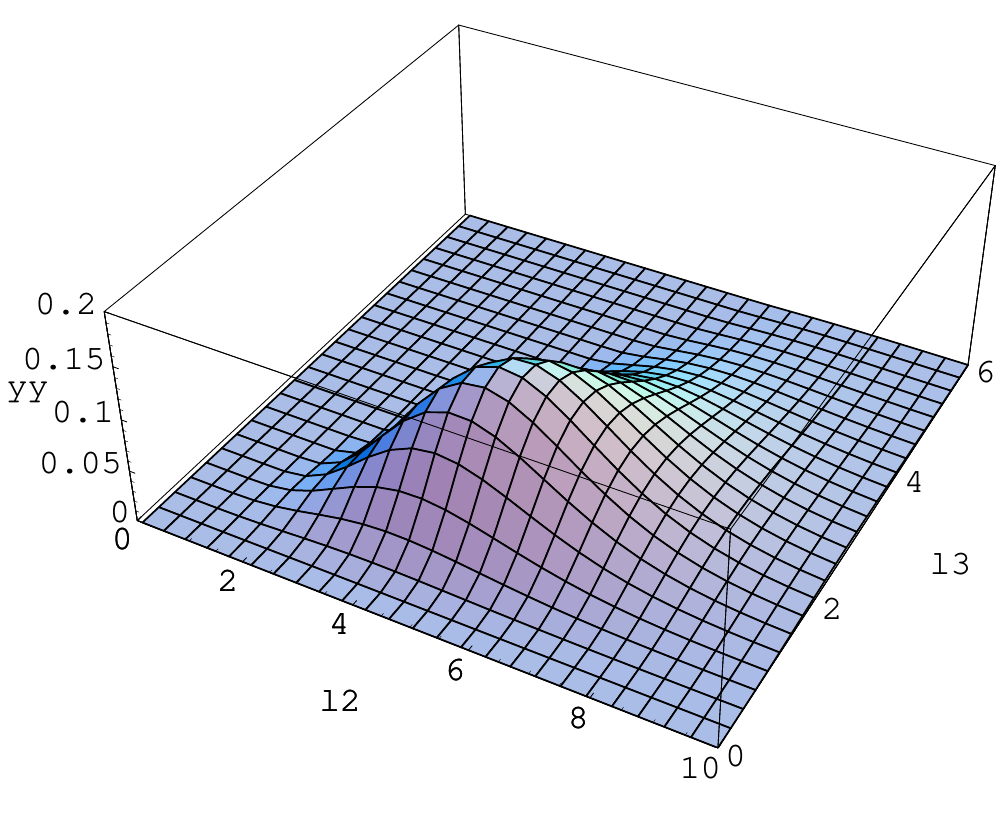}}
\caption{Joint probability density function of $\lambda_2$ and $\lambda_3$ of the uncorrelated central Wishart matrix with $M=4$ and $n=5$.} 
\label{fig:pdflambda23}
\end{figure}

\begin{figure}[h!]
\psfrag{yy}{}
\psfrag{l1}{$\lambda_1$}
\psfrag{l2}{$\lambda_2$}
\psfrag{l3}{$\lambda_3$}
\psfrag{l4}{$\lambda_4$}
\centerline{\includegraphics[width=7.5cm,draft=false]
    {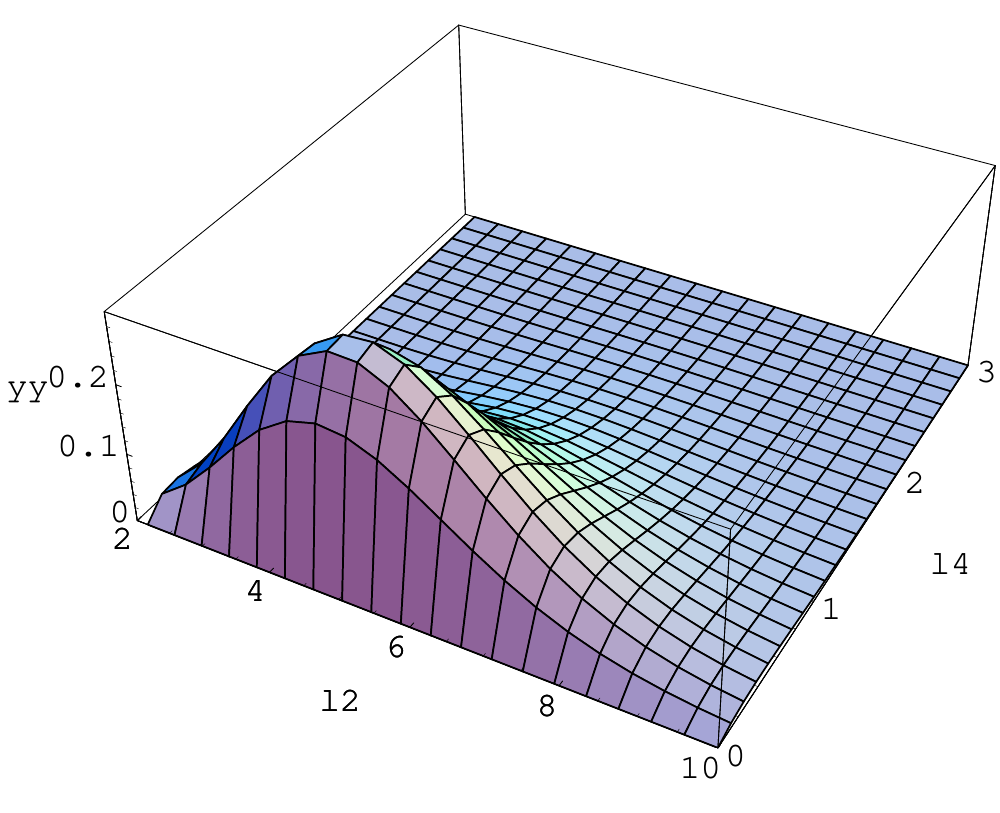}}
\caption{Joint probability density function of $\lambda_2$ and $\lambda_4$ of the uncorrelated central Wishart matrix with $M=4$ and $n=5$.} 
\label{fig:pdflambda24}
\end{figure}

\begin{figure}[h!]
\psfrag{yy}{}
\psfrag{l1}{$\lambda_1$}
\psfrag{l2}{$\lambda_2$}
\psfrag{l3}{$\lambda_3$}
\psfrag{l4}{$\lambda_4$}
\centerline{\includegraphics[width=7.5cm,draft=false]
    {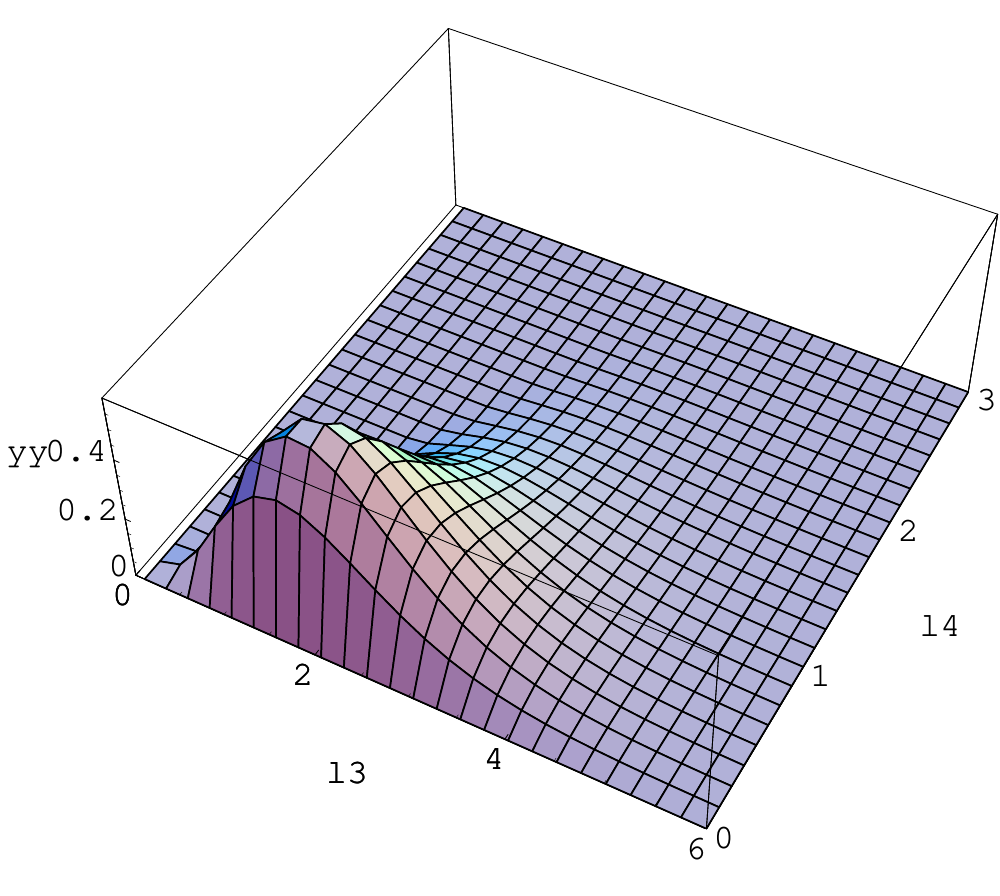}}
\caption{Joint probability density function of $\lambda_3$ and $\lambda_4$ of the uncorrelated central Wishart matrix with $M=4$ and $n=5$.} 
\label{fig:pdflambda34}
\end{figure}

\end{document}